\documentclass[a4paper,10pt]{article}
\usepackage[centertags]{amsmath}
\usepackage{amsfonts}
\usepackage{amssymb}
\usepackage{amsthm}
\usepackage{epsfig}
\usepackage{setspace}
\usepackage{ae}
\usepackage{eucal}
\usepackage[usenames]{color}%
\usepackage{graphicx}

\theoremstyle{plain}
\newtheorem{thm}{Theorem}[section]
\newtheorem{lem}[thm]{Lemma}

\theoremstyle{definition}

\theoremstyle{remark}

\newtheorem{rem}[thm]{Remark}

\title{On partial convolution and  mean comparison }

\author{J\"org Kampen}
\begin{document}

\maketitle

\begin{abstract}
We prove mean comparison results from a different perspective, where we introduce the concept of partial convolutions. For a parabolic initial data problem on the whole domain of dimension $n$ we consider data functions which live on a subspace of lower dimension $k$ and coefficient functions which live on the whole space or some subspace of dimension $l$. The pair of natural numbers  $(k,l)$ is called  a strong partial convolution pair if for some  coordinate transformations  the coefficients functions of the equation and the initial data functions live in complementary  linear  spaces  of dimension $k$ and $l$ respectively. Here the qualification 'strong' indicates that this transformation exists without any further restriction concerning the intersection of the original  linear subspaces involved, and that refined concepts are possible in this respect. For purely second order parabolic equations we show that $(1,n)$ for any $n\geq 2$ is a strong partial convolution pair. As a consequence new criteria for convexity preservation for classes of initial value problems are obtained. A further  consequence is mean comparison for univariate convex  functions of a considerable class of sums of locally continuous martingales.  We leave the problem of determination of all  partial convolution pairs as an open problem. In the literature it is observed that $(n,n)$ is not a strong partial convolution pair.

\end{abstract}

\section{Introduction}

This work is motivated by the problem of mean stochastic comparison for univariate convex functions  of sums of locally continuous martingales and the solution of related optimal control problems.  Classical stochastic analysis results such as Levy's characterisation of Brownian motions and representation theorems of local continuous martingales in terms of stochastic integrals with respect to Brownian motions show that Brownian motion is the fundamental continuous martingale. The former tells us that local $n$-dimensional martingales with cross variation matrix process equal to identity times time is indeed a $n$-dimensional Brownian motion with respect to the same filtration, and the latter tells us that local continuous martingales can be represented in terms of stochastic integrals with respect to Brownian motions where the integrand is measurable and adapted (to a possibly augmented filtration) such that the covariation matrix process is bounded and can be represented in terms of simple square form time integrals of the integrand components. Therefore  processes $X=(X_1,\cdots,X_n)$ of the form
\begin{equation}\label{eq1}
X(t)-X(0)=\int_0^t\sigma\left(X(s)\right)dW(s),
\end{equation}
with a $n$-dimensional  standard Brownian motion $W=(W_1,\cdots,W_n)$ (for construction cf. \cite{W}) and a (mildly regular)  matrix volatility function $\sigma(X)=\left(\sigma_{ij}(X)\right)_{1\leq i,j\leq n}$ represent a considerable class of  $n$-dimensional continuous martingales, which can be qualified to be 'special' only from a quite abstract point of view. The assumption of a quadratic volatility matrix is for simplicity and not essential.
For univariate convex data functions  $f$ we are interested in comparison results for stochastic sums of the form
\begin{equation}\label{eq2}
\sigma\sigma^T\leq \rho\rho^T\rightarrow E^x\left(f\left(\sum_{i=1}^nc_i X_i(t)\right) \right)\leq E^x\left(f\left(\sum_{i=1}^nc_i Y_i(t)\right) \right)
\end{equation}
where
\begin{equation}\label{eq3}
Y(t)-Y(0)=Y(t)-X(0)=\int_0^t\rho\left(Y(s)\right)dW(s),
\end{equation}
is another martingal diffusion, and $\sigma\sigma^T\leq \rho\rho^T$ means that  $ \rho\rho^T-\sigma\sigma^T$ is a nonnegative matrix.  Versions with strict inequalities in (\ref{eq2}) are also of interest, especially for solving optimal control problems.

 For example, in finance, options on  portfolio processes
\begin{equation}
\Pi_q=\sum_{i=1}^n q_i\sigma_iS_i,~\mbox{where}~\frac{dS}{S}=\sigma(S)dW,~S(0)=x\in {\mathbb R}^n
\end{equation} 
may be considered, where $\frac{dS}{S}=(\frac{dS_1}{S_1},\cdots,\frac{dS_n}{S_n})$ represent $n$ underlying assets and $\sigma$ is a matrix-valued volatility function. Here ${\mathbb R}^n$ denotes the $n$-dimensional Euclidean space, whee ${\mathbb R}$ denotes the field of real numbers.
This is a trading account with $n$ lognormal processes $\left(S_i \right)_{1\leq i\leq n}$, where  $q_i\in [-1,1]$ are bounded trading positions, where comparison shows under mild assumptions that any solution of an optimal control problem
\begin{equation}
\sup_{-1\leq q_i\leq 1,~ 1\leq i\leq n}E^x\left(f(\Pi_q) \right),~f~\mbox{convex, exponentially bounded},
\end{equation}
for the trading postions $q_i\in [-1,1]$ maximizes of the basket volatility, i.e.,
\begin{equation}
\sup_{-1\leq q_i\leq 1,~ 1\leq i\leq n}\frac{\sqrt{ <qS,\sigma\sigma^T(S)qS>}}{\sum_{i=1}^nS_i}.
\end{equation}
Here, $qS=(q_1S_1,\cdots,q_nS_n)^T$,  $\sigma^T$ is the transposed volatility matrix function, and $<.,.>$ denotes the scalar product. This is an example which shows that multivariate comparison results for stochastic sums are qualitative different from univariate results: in the latter case Passport options written on one asset can be subsumed by Lookback options, but multivariate passports are different from Lookback options. Control problems of this forms have been studied in the context of viscosity solution concept, where an overview can be found in \cite{Cr} and \cite{FS}. The classical theory is represented in \cite{Kryl1}. The application of passport options introduced in \cite{HLP}, is considered in \cite{DY}, \cite{HH}, \cite{SV}. Multivariate passport options are considered in \cite{KPP}.

Back to comparison  itself we add two remarks. The restriction to continuous martingales is natural since comparison for jump 
diffusions does not hold in general since it does not hold for simple Poisson processes as is shown in \cite{V}. Secondly, there are
rather immediate extensions of comparison results to semimartingales if the data are monoton in addition of being convex. As the same methods of proof given here applies to this case, we only mention this  rather obvious possibility of extension. 
We denote continuous functions on the field of real numbers by $C({\mathbb R})$, and consider the class of univariate convex  data functions $f\in C({\mathbb R})$ which satisfy the exponential growth condition
\begin{equation}\label{growthcd}
\mbox{for all}~x\in {\mathbb R}~|f(x)|\leq c\exp\left(c|x|^{2-\epsilon} \right) 
\end{equation}
for some constant $c>0$.
The purpose of this paper is to provide a different elementary proof of the main result in \cite{KC}. In the list of minimal assumptions for comparison we need the existence of continuous solutions of the stochastic differential equations describing the stochastic processes involved, which is ensured by bounded Lipschitz continuous volatility matrices. Furthermore the existence of the mean values  $ E^{x_0}\left(f\left(\sum_{i=1}^nc_i X_i(t)\right) \right), E^{x_0}\left(f\left(\sum_{i=1}^nc_i Y_i(t)\right) \right)$ for arbitrary starting points $x_0\in {\mathbb R}^n$ are needed. This involves certain growth conditions for the data function $f$, but also the existence of  solutions for the associated initial data problems. In this case we say that the mean value functions exist. In order to obtain strict inequalities we have to assume the existence of a positive $C^2$ density. This assumption holds certainly if an uniform ellipticity condition is satisfied, i.e.,
\begin{equation}
\begin{array}{ll}
(E): \exists \lambda,\Lambda \forall z\in {\mathbb R}^n, z\neq 0:~0<\lambda\leq \sum_{ij=1}^n(\sigma\sigma^T)_{ij}z_iz_j\leq \Lambda<\infty\\
\\
\hspace{1cm}\&~~ 0<\lambda\leq \sum_{ij=1}^n(\rho\rho^T)_{ij}z_iz_j\leq \Lambda<\infty.
\end{array}
\end{equation}
Related a priori estimates used in our argument can be found in standard references such as  \cite{Kryl2} , and \cite{LS}.
Positive $C^2$-densities can be also ensured for much weaker conditions, the so-called
\begin{equation}
(H):~\mbox{ H\"ormander condition (cf. \cite{KS}}).
\end{equation}
 The specification of the condition in (H) can be obtained from the latter reference. It is remarkable that the assumption in (H) is sufficient in order to get comparison with strict inequalities, because the upper bounds for spatial derivatives of the densities have no  decay at spatial infinity in general. However, for the density itself we have Gaussian upper bounds and it turns out that this suffices if the mean value functions are finite.
\begin{thm}\label{main1}
Let $T>0$, $t\in [0,T]$,  $f\in C({\mathbb R})$ be convex, and assume that $f$ satisfies the exponential growth condition in (\ref{growthcd}). Assume that $c_i>0$ are some positive real constants for $1\leq i\leq n$. Furthermore, let $X,Y$ be It\^o's diffusions with $x_0=X(0)=Y(0)$, where
\begin{equation}
X(t)=X(0)+\int_0^t\sigma\left(X(s)\right)dW(s),
\end{equation}
\begin{equation}
Y(t)=Y(0)+\int_0^t\rho\left(Y(s)\right)dW(s),
\end{equation}
with $n\times n$-matrix valued bounded Lipschitz-continuous functions $x\rightarrow \sigma\sigma^T(x)$ and $y\rightarrow \rho\rho^T$. If $\sigma\sigma^T\leq \rho\rho^T$, then for $0\leq t\leq T$ we have
\begin{equation}
E^{x_0}\left(f\left(\sum_{i=1}^nc_i X_i(t)\right) \right)\leq E^{x_0}\left(f\left(\sum_{i=1}^nc_i Y_i(t)\right) \right),
\end{equation}
where we assume that the mean value functions exist . Here, the relation  symbol $\leq$ for matrices refers to the usual order of positive matrices. Furthermore, if in addition $ f''\neq 0$ (in  the sense of distributions) and condition (E) or the weaker condition  (H) holds, then this result holds with strict inequalities.
If in the latter case the density is positive on the time interval $[0,T]$, then   the value function $(t,x_0)\rightarrow v(t,x_0)= E^{x_0}\left(f\left(\sum_{i=1}^nc_i X_i(t)\right) \right)$ is strictly monoton with respect to time in the whole time interval  $[0,T]$.
\end{thm}
 This theorem is considered from a different perspective in \cite{KC}. Possibilities of generalisations are limited in the sense that $(n,n)$ is not a strong partial convolution pair as is shown in \cite{T}. Furthermore, the assumption of continuous processes is essential as is shown in \cite{V}. The univariate form of the theorem was proved in \cite{Ha}. 
\section{Proof of Main Theorem}

\begin{proof} First we prove strict monotonicity with respect to time, where we assume that $0\neq f$ is convex and a strict inequality holds with respect to the coefficient matrices, i.e.,  $\sigma\sigma^T< \rho\rho^T$. Let $f^{\epsilon}\in C^{\infty}\cap H^2$ be a mollification of $f$ which approximates $f$ on an arbitrarily large compact set $K$, where $H^2$ is the standard Sobolev space of order two, i.e., weak partial derivatives up to order two are in $L^2$ . In order to achieve this we may convolute  the data $f$ with a heat kernel for small time  and multiply the convoluted data by a spatial damping factor $\exp(-\epsilon |x|^2)$ with a small parameter $\epsilon >0$. Here $|.|$ denotes the Euclidean norm in ${\mathbb R}^n$. Refinement of this construction with a damping factor which equals $1$ on an arbitrarily large compact domain containing $K\subset  {\mathbb R}^n$ leads to strict convexity of approximating data  $f^{\epsilon}$ on the compact set $K$, i.e., positive definiteness of the Hessian on $K$. Note that this $K$ can be an arbitrarily large compact set. We note that we may transform to an equivalent system with the stochastic sum variable $Z(t)=\sum_{i=1}^nc_i X_i(t)$ such that we may represent  the value function $E^{x}\left(f\left(Z(t)\right) \right)$ by the solution of a multivariate parabolic equation with univariate data. We take this as a starting point and assume w.l.o.g. that such a transformation is performed keeping the same notation for the volatilities.
If $(a_{ij}):=\sigma\sigma^T$ is sufficiently regular and satisfies a usual ellipticity condition, and the initial data satisfy an exponential growth condition then the Feynman-Kac formalism tells us that the value function $v$ satisfies the initial value problem
\begin{equation}\label{initprob}
Lv\equiv v_t-\sum_{ij}a_{ij}v_{x_ix_j}=0,~v(0,x)=f(x_1),
\end{equation}
where $x=(x_1,\cdots,x_n)$ and the univariate function $f$ depends on the component $x_1$ without loss of generality. 
Since $\sigma\sigma^T\geq 0$ is bounded Lipschitz the assumption of the latter sentence is practically not restrictive as such functions may be approximated on bounded domains by strictly elliptic matrix functions $(a_{ij})>0$ with bounded $C^{2,b}$ coefficients $a_{ij}$, where $C^{2,b}$ denotes the function space of bounded functions with bounded continuous derivatives up to second order.  Let $v^{\epsilon}$ denote the solution of the initial value problem in (\ref{initprob}) with univariate  data $f^{\epsilon}$ approximating the data $f$ on a $K$ as described above.
For $v^{f^{\epsilon}}:=v^{\epsilon}-f^{\epsilon}$ we have 
\begin{equation}
Lv^{f^{\epsilon}}=a_{11}f^{\epsilon}_{x_1x_1}>0\mbox{ on a large domain $K$},
\end{equation}
since $a_{11}(x)>0$ and $f^{\epsilon}>0$ on such a domain $K$.
For all $(t,x)\in {[0,t]}\times{\mathbb R}^n$ we have
\begin{equation}
v^{f^{\epsilon}}(t,x)=\int_{{[0,t]}\times{\mathbb R}^n}
\left(a_{11}f^{\epsilon}_{xx}\right)(y)p(t,x;s,y)dyds,
\end{equation}
where $p>0$ is the fundamental solution of the parabolic equation in (\ref{initprob}). It follows that 
\begin{equation}
t_1<t_2\rightarrow v^{f^{\epsilon}}(t_1,x)<v^{f^{\epsilon}}(t_2,x)
\end{equation}
For each $x\in {\mathbb R}^n$ fixed this holds also in the limit $\epsilon \downarrow 0$. 

Next we prove comparison. We consider  all affine coordinate  transformations of the initial value problem in \ref{initprob} , i.e., transformations of the form $x\rightarrow z=c+Dx$, where  $c\in {\mathbb R}^n$ is a constant vector and $D$ is an invertible  $n\times n$-matrix of constants  such that (\ref{initprob}) becomes 
\begin{equation}\label{initprob2}
L^{c,D}v^{c,D}\equiv v^{c,D}_t-\sum_{ij}a^{c,D}_{ij}v^{c,D}_{z_iz_j}=0,~v^{c,D}(0,z)=f(z_1)=f(c+Dx_1).
\end{equation} 
It is sufficient to prove that $v^{c,D}_{z_1z_1}(t,.)>0$ holds for $t>0$  for all vectors $c$ and invertible matrices $D$. This implies global convexity where we  recall that a function $h$ is convex iff
\begin{equation}
\forall y,z\in {\mathbb R}^n~\forall \lambda\in [0,1]:~h(\lambda y+(1-\lambda)z)\leq \lambda h(y)+(1-\lambda)h(z).
\end{equation}
Since the transformed problems in (\ref{initprob2}) have the same structure as the original problem in  (\ref{initprob2}), it is essential to prove $v_{x_1x_1}(t,.)>0$ for $t>0$.  

We consider the essential case $n=2$, where an analogous argument holds in the case $n>2$.

We consider a general  transformation
\begin{equation}
v(t,x)=u(t,y),~\mbox{for}~x=(x_1,x_2),~y=(y_1,y_2),
\end{equation}
where $y\equiv y(x)$ is a smooth coordinate transformation.
For the first order derivatives we have
\begin{equation}
v_{x_1}=u_{y_1}\frac{\partial y_1}{\partial x_1}+u_{y_2}\frac{\partial y_2}{\partial x_1},
\end{equation}
and for the second order derivatives we have
\begin{equation}
\begin{array}{ll}
v_{x_1x_1}=u_{y_1y_1}\frac{\partial y_1}{\partial x_1}\frac{\partial y_1}{\partial x_1}+2u_{y_1y_2}\frac{\partial y_1}{\partial x_1}\frac{\partial y_2}{\partial x_1}+u_{y_2y_2}\frac{\partial y_2}{\partial x_1}\frac{\partial y_2}{\partial x_1}\\
\\
+u_{y_1}\frac{\partial^2 y_1}{\partial x_1^2}+u_{y_2}\frac{\partial^2 y_2}{\partial x_1^2}
\end{array}
\end{equation}
Analogous equalities hold for
$v_{x_1x_2}$ $v_{x_2x_2}$ 
(with   $\frac{\partial^2}{\partial x_1\partial x_2}$ -derivatives and  $\frac{\partial^2}{\partial x_2^2}$-derivatives respectively). This leads to the transformed equation
\begin{equation}\label{trfeq}
\begin{array}{ll}
u_t-u_{y_1y_1}\left(a_{11}\left(\frac{\partial y_1}{\partial x_1}\right)^2+a_{12}\frac{\partial y_1}{\partial x_1}\frac{\partial y_1}{\partial x_2}+a_{22}\left(\frac{\partial y_1}{\partial x_2}\right)^2\right)\\
\\
-u_{y_1y_2}\left(2a_{11}\frac{\partial y_1}{\partial x_1}\frac{\partial y_2}{\partial x_1}+2a_{12}\frac{\partial y_1}{\partial x_1}\frac{\partial y_2}{\partial x_2}+2a_{22}\frac{\partial y_1}{\partial x_2}\frac{\partial y_2}{\partial x_2}\right)\\
\\
-u_{y_2y_2}\left(a_{11}\left(\frac{\partial y_2}{\partial x_1}\right)^2+a_{12}\frac{\partial y_2}{\partial x_1}\frac{\partial y_2}{\partial x_2}+a_{22}\left(\frac{\partial y_2}{\partial x_2}\right)^2\right)\\
\\
-u_{y_1}\left(a_{11}\frac{\partial^2 y_1}{\partial x_1^2}+a_{12}\frac{\partial^2 y_1}{\partial x_1\partial x_2}+a_{22}\frac{\partial^2 y_1}{\partial x_2^2}\right)\\
\\
-u_{y_2}\left(a_{11}\frac{\partial^2 y_2}{\partial x_1^2}+a_{12}\frac{\partial^2 y_2}{\partial x_1\partial x_2}+a_{22}\frac{\partial^2 y_2}{\partial x_2^2}\right)=0
\end{array}
\end{equation}
We look for a coordinate  transformation   $(x_1,x_2)\rightarrow y_1(x_1,x_2),~ (x_1,x_2)\rightarrow y_2(x_1,x_2)$,  and univariate functions $c_{11}\equiv c_{11}(y_2)$, $c_{12}\equiv c_{12}(y_2),$ $c_{22}\equiv c_{22}(y_2),~c_1\equiv c_1(y_2),~ c_2\equiv c_2(y_2)$ such that we can approximatively solve the equations (of the argument $x=(x_1,x_2)$)
\begin{equation}\label{eqns}
\begin{array}{ll}
c_{11}=\left(a_{11}\left(\frac{\partial y_1}{\partial x_1}\right)^2+a_{12}\frac{\partial y_1}{\partial x_1}\frac{\partial y_1}{\partial x_2}+a_{22}\left(\frac{\partial y_1}{\partial x_2}\right)^2\right)\\
\\
c_{12}=\left(2a_{11}\frac{\partial y_1}{\partial x_1}\frac{\partial y_2}{\partial x_1}+2a_{12}\frac{\partial y_1}{\partial x_1}\frac{\partial y_2}{\partial x_2}+2a_{22}\frac{\partial y_1}{\partial x_2}\frac{\partial y_2}{\partial x_2}\right)\\
\\
c_{22}=\left(a_{11}\left(\frac{\partial y_2}{\partial x_1}\right)^2+a_{12}\frac{\partial y_2}{\partial x_1}\frac{\partial y_2}{\partial x_2}+a_{22}\left(\frac{\partial y_2}{\partial x_2}\right)^2\right)\\
\\
c_1=\left(a_{11}\frac{\partial^2 y_1}{\partial x_1^2}+a_{12}\frac{\partial^2 y_1}{\partial x_1\partial x_2}+a_{22}\frac{\partial^2 y_1}{\partial x_2^2}\right)\\
\\
c_2=\left(a_{11}\frac{\partial^2 y_2}{\partial x_1^2}+a_{12}\frac{\partial^2 y_2}{\partial x_1\partial x_2}+a_{22}\frac{\partial^2 y_2}{\partial x_2^2}\right).
\end{array}
\end{equation}
Here, by approximatively we mean that  the equations in (\ref{eqns}) can be solved in  $C^2(K)$ for  an arbitrary compact set $K$ up to any small positive real number $\epsilon >0$  with respect to the classical norm $\|.\|_{C^2 (K)}$ where the approximative solution function can be extended to a  function in $H^2$. This is proved in Lemma \ref{lem1} below.
With this choice of $y_1,y_2$, and $c_{11},~c_{12},~c_{22},~c_1,~c_2$ we have 
\begin{equation}\label{ueq}
\begin{array}{ll}
u_t=c_{11}(y_2)u_{y_1y_1}+c_{12}(y_2)u_{y_1y_2}+c_{22}(y_2)u_{y_2y_2}+c_1(y_2)u_{y_1}
+c_2(y_2)u_{y_2}.
\end{array}
\end{equation}
In order to show convexity we can argue as above that it is essential  to prove convexity with respect to a variable $y_1$, i.e., $u_{y_1y_1}(t,.)>0$  for $t>0$. This follows form a refinement of the argument above where we consider affine transformations which leave the $y_2$-component untouched (see below).
Using the Levy expansion for the construction of fundamental solutions of parabolic equations we show in  Lemma \ref{lem2} below that the   density or fundamental solution of (\ref{ueq}) has a representation of the form
\begin{equation}
(t,y_1,y_2;s,z_1,z_2)\rightarrow p(t,y_1-z_1,y_2;s,z_2)\mbox{, i.e.,}
\end{equation}
the solution function is a convolution with respect to the first variable. We denote $u^{\epsilon}(t,y)=v^{\epsilon}(t,x)$ for all $t>0$ and $x,y\in {\mathbb R}^n$.
Then $u^{\epsilon}$ has the representation
\begin{equation}
\begin{array}{ll}
u^{\epsilon}(t,y)=\int_{{\mathbb R}^2}g^{\epsilon}(z_1,z_2)p(t,y_1-z_1,y_2;0,z_2)dz\\
\\
=\int_{{\mathbb R}^2}g^{\epsilon}(y_1-z_1,z_2)p(t,y_1,y_2;0,z_2)dz
\end{array}
\end{equation}
where the convolution rule holds with respect to the first variable, and
\begin{equation}
g^{\epsilon}(y_1,y_2)=f^{\epsilon}(x_1)
\end{equation}
is defined via the coordinate transformation $(x_1,x_2)\rightarrow (y_1(x_1,x_2),y_2(x_1,x_2))$. We note that $g\equiv \lim_{\epsilon\downarrow 0} g^{\epsilon}$ is a convex function under arbitrary coordinate transformations, because convexity is preserved under coordinate transformations.
Hence, we have
\begin{equation}
\begin{array}{ll}
u^{\epsilon}_{y_1y_1}(t,y)=\int_{{\mathbb R}^2}g^{\epsilon}(z_1,z_2)p_{y_1y_1}(t,y_1-z_1,y_2;0,z_2)dz\\
\\
=\int_{{\mathbb R}^2}g^{\epsilon}_{y_1y_1}(y_1-z_1,z_2)p(t,z_1,y_2;0,z_2)dz,
\end{array}
\end{equation}
and this relation holds also in the limit $\epsilon \downarrow 0$.
Here we note that $g_{y_1y_1}$ exists almost everywhere according to \cite{A}.
This argument holds for all affine coordinate  transformations $z=c+Dy_1$ (which leaves the $y_2$-axis untouched) such that   $u_{z_1z_1}(t,.)>0$ for all $t>0$. Note that all matices $D$ are considered such that $(c+Dy_1,y_2)$ is a coordinate transformation and these are almost all, which is sufficient.  Hence $u$ is convex, and  we conclude that $v$ is convex. 
\end{proof}

\begin{lem}\label{lem1}
The equation system in (\ref{eqns})with functions $c_{11},~c_{12}~,c_{22},~c_1,~c_2$ dependent only on $y_2$ can be solved approximatively on an arbitrarily large compact domain $K$ with respect to a $C^2$-norm, and such that this approximative solution can be extended in $H^2\cap C^2$ space to the whole domain of ${\mathbb R}^n$.
\end{lem}

\begin{proof}
First, we reduce the partial differential equation system 
\begin{equation}\label{eqns2l}
\begin{array}{ll}
c_{11}=\left(a_{11}\left(\frac{\partial y_1}{\partial x_1}\right)^2+a_{12}\frac{\partial y_1}{\partial x_1}\frac{\partial y_1}{\partial x_2}+a_{22}\left(\frac{\partial y_1}{\partial x_2}\right)^2\right)\\
\\
c_1=\left(a_{11}\frac{\partial^2 y_1}{\partial x_1^2}+a_{12}\frac{\partial^2 y_1}{\partial x_1\partial x_2}+a_{22}\frac{\partial^2 y_1}{\partial x_2^2}\right)
\end{array}
\end{equation}
to a nonlinear ordinary differential equation for $\frac{dy_1}{dx_1}$. Here the first equation in (\ref{eqns2l}) is solved for $\frac{\partial y_1}{\partial x_2}$ first. We get a solution
\begin{equation}\label{eqy12}
\frac{\partial y_1}{\partial x_2}=-\frac{a_{12}}{a_{22}}\frac{\partial y_1}{\partial x_1}+\sqrt{\frac{c_{11}}{a_{22}}-\frac{a_{11}}{a_{22}}\left(\frac{\partial y_1}{\partial x_1}\right)^2+\frac{1}{4}\left( \frac{a_{12}}{a_{22}}\frac{\partial y_1}{\partial x_1}\right)^2}
\end{equation}
Next we re-express the second equation in (\ref{eqns2l}) as a nonlinear differential equation of $\frac{\partial y_1}{\partial x_1}$  and $\frac{\partial^2 y_1}{\partial x_1^2}$ which is an ordinary differential equation with parameter $x_2$, essentially. We have to eliminate the expressions $\frac{\partial^2 y_1}{\partial x_1\partial x_2}$ and $\frac{\partial^2 y_1}{\partial x_2^2}$.
Differentiating (\ref{eqy12}) with respect to $x_1$ we get  
 \begin{equation}\label{eqy121}
\frac{\partial^2 y_1}{\partial x_1\partial x_2}=-\frac{\partial}{\partial x_1}\left( \frac{a_{12}}{a_{22}}\frac{\partial y_1}{\partial x_1}+\sqrt{\frac{c_{11}}{a_{22}}-\frac{a_{11}}{a_{22}}\left(\frac{\partial y_1}{\partial x_1}\right)^2+\frac{1}{4}\left( \frac{a_{12}}{a_{22}}\frac{\partial y_1}{\partial x_1}\right)^2}\right), 
\end{equation}
where the right expression in (\ref{eqy121}) can be expanded such that $\frac{\partial^2 y_1}{\partial x_1\partial x_2}$ can be expressed as a functional of $\frac{\partial^2 y_1}{\partial x_1^2}$ and of $\frac{\partial y_1}{\partial x_1}$. We do not need to expand and write down the special form of this expansion, but we should remark that we may choose the value of $c_{11}(y_2(x_1,x_2))$ large enough on the compact set $K$ that there are no singularities (still sustaining degrees of freedom in choosing derivatives of $c_{11}$). Furthermore we remark that the right expression in (\ref{eqy121}) we get on expansion is linear in $\frac{\partial y_1^2}{\partial x_1^2}$ and nonlinear in $\frac{\partial y_1}{\partial x_1}$.
Next differentiating (\ref{eqy12}) with respect to $x_2$ we get 
\begin{equation}\label{eqy1212}
\frac{\partial^2 y_1}{\partial x_2^2}=-\frac{\partial}{\partial x_2}\left( \frac{a_{12}}{a_{22}}\frac{\partial y_1}{\partial x_1}+\sqrt{\frac{c_{11}}{a_{22}}-\frac{a_{11}}{a_{22}}\left(\frac{\partial y_1}{\partial x_1}\right)^2+\frac{1}{4}\left( \frac{a_{12}}{a_{22}}\frac{\partial y_1}{\partial x_1}\right)^2}\right), 
\end{equation}
where the right expression in (\ref{eqy1212}) can be expanded such that $\frac{\partial^2 y_1}{\partial x_2^2}$ can be expressed as a functional of $\frac{\partial^2 y_1}{\partial x_1\partial x_2}$ and of $\frac{\partial y_1}{partial x_1}$. Again, we do not need to expand and write down the special form of this expansion, but  remark that we may choose the value of $c_{11}(y_2(x_1,x_2))$ large enough on the compact set $K$ such that there are no singularities (still sustaining degrees of freedom in choosing derivatives of $c_{11}$). Furthermore we remark that the right expression in (\ref{eqy1212}) we get on expansion is linear in $\frac{\partial y_1^2}{\partial x_1\partial x_2}$ and nonlinear in $\frac{\partial y_1}{\partial x_1}$. 
Next we can substitute the expressions in (\ref{eqy12}), (\ref{eqy121}), and (\ref{eqy1212}) into the second equation of equation system (\ref{eqns2l}). The mixed derivatives obtained by using (\ref{eqy1212}) in the latter substitution can be substituted again using the equation in (\ref{eqy121}) a second time, and we get get a nonlinear (essentially) ordinary differential equation. The latter equation is still linear  $\frac{\partial y_1^2}{\partial x_1^2}$ and nonlinear in $\frac{\partial y_1}{\partial x_1}$, and can be solved for $\frac{\partial y_1}{\partial x_1}$ in a solvable nonlinear integral equation form which is globally solvable.
Similarly, we reduce the  partial differential equation 
\begin{equation}
\label{eqsecond}
\begin{array}{ll}
c_{22}=\left(a_{11}\left(\frac{\partial y_2}{\partial x_1}\right)^2+a_{12}\frac{\partial y_2}{\partial x_1}\frac{\partial y_2}{\partial x_2}+a_{22}\left(\frac{\partial y_2}{\partial x_2}\right)^2\right)\\
\\
c_2=\left(a_{11}\frac{\partial^2 y_2}{\partial x_1^2}+a_{12}\frac{\partial^2 y_2}{\partial x_1\partial x_2}+a_{22}\frac{\partial^2 y_2}{\partial x_2^2}\right).
\end{array}
\end{equation}
to a nonlinear ordinary differential equation for $\frac{dy_2}{dx_2}$. We indicate dependence of solutions $(x_1,x_2)\rightarrow y^{c_{11},c_1}_1(x_1,x_2)$ and $(x_1,x_2)\rightarrow y^{c_{22},c_2}_2(x_1,x_2)$ of the equation  systems (\ref{eqns2l}) and  (\ref{eqsecond}) 
respectively by upper scripts. 
We can choose a function $(x_1,x_2)\rightarrow y^{c_{22},c_2}_2(x_1,x_2)$ such that 
\begin{equation}\label{coordinfo}
\frac{\partial y^{c_{22},c_2}_2}{\partial x_1}>0,~\frac{\partial y^{c_{22},c_2}_2}{\partial x_2}>0
\end{equation}
holds at all arguments $x_1,x_2\in K$.
Hence the mixed term equation takes the form
\begin{equation}\label{mixed}
\begin{array}{ll}
c_{12}(y^{c_{22},c_2}_2(x_1,x_2))=a_{11}(x_1,x_2)\frac{\partial y^{c_{11},c_1}_1}{\partial x_1}(x_1,x_2)\frac{\partial y^{c_{22},c_2}_2}{\partial x_1}(x_1,x_2)\\
\\
+a_{12}(x_1,x_2)\frac{\partial y^{c_{11},c_1}_1}{\partial x_1}(x_1,x_2)\frac{\partial y^{c_{22},c_2}_2}{\partial x_2}(x_1,x_2)\\
\\
+a_{22}(x_1,x_2)\frac{\partial y^{c_{11},c_1}_1}{\partial x_2}(x_1,x_2)\frac{\partial y^{c_{22},c_2}_2}{\partial x_2}(x_1,x_2).
\end{array}
\end{equation}
Using (\ref{coordinfo}) we can use two functions $c_{11},c_{22}$ as degrees of freedom in order to solve (\ref{mixed}) approximately in $C^2$-norm on $K$, where polynomial approximations and the methods in \cite{Kpre}.

\end{proof}

\begin{lem}\label{lem2}
The fundamental solution of (\ref{ueq}) has the representation
\begin{equation}
\begin{array}{ll}
 p(t,x_1-y_1,x_2;s,y_2)=G(t-s,x;y)+\\
\\
+\int_s^t\int_{{\mathbb R}^n}\left(\sum_{m=1}^{\infty}L^{(y_2)}_{m}G(t,x;s,y)\right)G(t-s,x;y)dyds,
\end{array}
\end{equation}
where the Gaussian is given by
\begin{equation}
G(t-s,x;y)=\frac{1}{\sqrt{4\pi(t-s)}^n}\exp\left(-\frac{N(x;y)}{(t-s)}\right)
\end{equation}
along with
\begin{equation}
N(x;y)=a_{11}(y_2)(x_1-y_1)^2+a_{12}(y_2)(x_1-y_1)(x_2-y_2)+a_{22}(y_2)(x_2-y_2)^2,
\end{equation}
and, inductively,
\begin{equation}
L^{(y_2)}_1G(t-s,x;y)=\sum_{i,j=1}^n(a_{ij}(x_2)-a_{ij}(y_2))\frac{\partial^2}{\partial x_i\partial x_j}G(t-s,x;y),
\end{equation}
and for $m\geq 1$
\begin{equation}
L^{(y_2)}_{m+1}G(t-s,x;y)=\int_s^t\int_{{\mathbb R}^n}L^{(y_2)}_{m}G(t-\sigma,x;z)L^{(y_2)}_{1}G(\sigma-s,z;y)dzd\sigma.
\end{equation}
\end{lem}

\begin{proof}
Follows from the Levy expansion of the fundamental solution of (\ref{ueq}).
\end{proof}

\begin{rem}
The preceding article is based  on further unpublished notes  from my  Lecture  \begin{center}'Die Fundamentall\"{o}sung parabolischer Gleichungen und schwache Schemata h\"{o}herer Ordnung f\"{u}r stochastische Diffusionsprozesse'
\end{center} 
 of WS 2005/2006 in Heidelberg.
\end{rem}

\begin{rem}
The methods in \cite{Kpre} are numerically tested, the paper was never submitted to a journal, but may be submitted in the future after more numerical tests are completed.
\end{rem}

\end{document}